\documentclass[11pt]{article}

\usepackage[cp1252]{inputenc}
\usepackage[english]{babel}
\usepackage[a4paper]{geometry}
\usepackage{amsmath,amsfonts,amssymb,amsthm,cases,upgreek}
\usepackage{empheq}
\usepackage{stmaryrd}
\usepackage{dsfont}
\usepackage{graphicx}
\usepackage{comment}
\usepackage{bm} 

\newtheorem{thm}{Theorem}[section]

\newtheorem{lemma}[thm]{Lemma}

\newtheorem{prop}[thm]{Proposition}
\newtheorem{remark}[thm]{Remark}

\newcommand{\lver}{\left\lvert}

\newcommand{\rver}{\right\rvert}

\renewcommand{\footnote}[1]{\textsuperscript{\addtocounter{footnote}{1}(\thefootnote)}\footnotetext{#1}}

\def \epsilon {\varepsilon}

\SetSymbolFont{stmry}{bold}{U}{stmry}{m}{n}
\DeclareMathOperator{\sgn}{sgn}

\begin{document}
\title{\textbf{Existence and stability of steady compressible Navier-Stokes solutions on a finite interval
with noncharacteristic boundary conditions}}
\author{Benjamin Melinand\footnote{Indiana University. Email : melinand@ceremade.dauphine.fr}, Kevin Zumbrun\footnote{Indiana University. Email : kzumbrun@indiana.edu}}
\date{October 2017}

\maketitle

\vspace{-1cm}

\begin{abstract}
We study existence and stability of steady solutions of the isentropic compressible Navier-Stokes equations on a finite interval with noncharacteristic boundary conditions, 
for general not necessarily small-amplitude data.
We show that there exists a unique solution,
about which the linearized spatial operator possesses (i) a spectral gap between neutral and growing/decaying modes, 
and (ii) an even number of nonstable eigenvalues $\lambda$ (with a nonnegative real part).
In the case that there are no nonstable eigenvalues, i.e., of spectral stability,
we show this solution to be nonlinearly exponentially stable in $H^2\times H^3$.
Using ``Goodman-type'' weighted energy estimates, we establish spectral stability 
for small-amplitude data.
For large-amplidude data, we obtain high-frequency stability, 
reducing stability investigations to a bounded frequency regime.  
On this remaining, bounded-frequency regime,
we carry out a numerical Evans function study, with results again indicating universal stability
of solutions.
\end{abstract}

\section{Introduction}

In this paper, we initiate in the simplest setting of 1D isentropic gas dynamics,
a systematic study of existence and stability of steady solutions of systems of hyperbolic parabolic
equations on a bounded domain, with noncharacteristic inflow or outflow boundary conditions,
and data and solutions of amplitudes that are not necessarily small.
We have in mind the scenario of a ``shock tube'', or finite-length channel with 
inflow-outflow boundary conditions, 
which in turn could be viewed as a generalization
of the Poisseuille flow in the incompressible case.

\medskip

Our conclusions in the present, isentropic case, obtained by rigorous nonlinear and spectral stability
theory, augmented in the large-amplitude case by numerical Evans function analysis, 
are that for any choice of data there exists a unique solution, and this solution is 
linearly and nonlinearly time-exponentially stable in $H^2\times H^3$.
These results suggest a number of interesting directions for further investigation in 1 and multi-D.

\subsection{Setting}
We consider the 1D isentropic compressible Navier-Stokes equations 
\begin{equation}\label{NS}
\left\{
\begin{aligned}
&\rho_{t} + \left(\rho u \right)_{x} = 0,\\
&\left(\rho u \right)_{t} + \left(\rho u^{2} + P(\rho) \right)_{x} = \nu u_{xx}
\end{aligned}
\right.
\end{equation}
on the interval $[0,1]$, with the noncharacteristic boundary conditions 
\begin{equation}\label{BC}
\left\{
\begin{aligned}
&\rho(t,0)=\rho_{0} > 0,\\
&u(t,0)=u_{0} > 0,\\
&u(t,1)=u_{1} > 0.
\end{aligned}
\right.
\end{equation}
Notice that we have an inflow boundary condition at $x=0$ and an outflow boundary condition at $x=1$. We assume that the viscosity $\nu$ is positive and constant and that the pressure $P$ is a smooth function satisfying

\begin{equation}\label{pressure_cond}
P' > 0.
\end{equation}

Stability of steady states for hyperbolic parabolic systems has been studied by many authors. For
problems on the whole line, the reader can refer to \cite{mascia_zumbrun,num_stab_zum} and references within. In the case of noncharacteristic boundary conditions on the half line, see for instance \cite{toan_zumbrun_ns,toan_zumbrun_hyp_par}. 
For studies of scalar conservation laws on a bounded interval, one may
see for instance \cite{Kreiss_Kreiss_burgers,Jiu_Pan_scalar_cons_bounded}. 
Finally, we refer to  \cite{control_ns_1d} for the study of boundary controllability of the 1D Navier-Stokes equations.

\medskip

In this paper, we study the existence and stability of steady states of \eqref{NS} satisfying the boundary conditions \eqref{BC}. Section \ref{section_existence_steady} is devoted to the existence and the uniqueness of such steady states. In Section \ref{section_linear_estimate}, we study the corresponding linearized problem about the steady state. In section \ref{section_spectral_stab}, we show that constant steady states and almost constant steady states, see Condition \eqref{almost_constant_assump}, are spectrally stable. We also show that general steady states are numerically spectrally stable. Section \ref{section_local_existence} is devoted to a local wellposedness result for problem \eqref{NS}-\eqref{BC}. Then, in section \ref{section_nonlinear_stab}, we show the nonlinear stability of steady states that are spectrally stable. Theorem \ref{stab_result} is the main result of this paper. Finally, in section \ref{section_improvement}, we improve the previous theorem under more restrictive assumptions. 

\begin{remark}\label{bcrmk}
It is worth noting that boundary conditions \eqref{BC} are not the only ones we can deal with. For instance, the case
\begin{equation*}
\left\{
\begin{aligned}
&\rho(t,1)=\rho_{0} > 0,\\
&u(t,0)=u_{0} < 0,\\
&u(t,1)=u_{1} < 0,
\end{aligned}
\right.
\end{equation*}
is equivalent by the change of variables $x \shortrightarrow 1-x$ and the change of unknowns $(\hat{\rho}, \hat{u}) \shortrightarrow (\hat{\rho}, -\hat{u})$.
Moreover, these two possibilities are the only types of noncharacteristic boundary conditions
yielding physically realizable steady states.
For, the first equation of \eqref{NS} yields that steady solutions have constant momentum
$\rho u\equiv m$, so that $u(0) $ and $u(1)$ necessarily agree in sign.
By similar reasoning, 
characteristic boundary conditions $u(0)=u(1)=0$ yield $u\equiv 0$
yield only trivial, constant steady states $(\rho, u)\equiv (\rho_0, 0)$.
\end{remark}

\subsection{Discussion and open problems}\label{s:discussion}
%
%

As mentioned earlier, our goal in this paper is to open a line of investigation of large-amplitude steady solutions for inflow-outflow problems on bounded domains.
The main technical contribution is our argument for nonlinear exponential stability of spectrally stable solutions, which is
both particularly simple and also applies to general hyperbolic parabolic systems of ``Kawashima type'', 
as considered on the whole- and half-line in \cite{mascia_zumbrun,num_stab_zum,toan_zumbrun_ns,toan_zumbrun_hyp_par}. 
Our goal, and the novelty of the argument as compared to those for the whole- and half-line, was to take advantage of the spectral
gap to obtain a simple proof based on standard semigroup/energy methods.
However, a close reading will reveal that this is deceptively difficult to accomplish, involving the introduction of a
precisely chosen space $(\rho, u)\in H^1 \times L^2$ with norm strong enough that we can carry out energy-based high-frequency 
resolvent estimates  and different from the usual Kawashima type estimates, but weak enough that the range of nonlinear terms is densely
contained.
\medskip

The reduction of nonlinear to spectral stability gives a base for investigation of more general systems such as full
(nonisentropic) gas dynamics or (isentropic or nonisentropic) MHD.
Our results on uniqueness and universal stability on the other hand are likely accidents of low dimension.
For example, the demonstration of unstable large-amplitude boundary layers in \cite{Serre_Zumbrun, zumbrun_standingshock}
is suggestive via the large-interval length limit from bounded interval toward the half-line, that 
unstable large-amplitude steady solutions might occur on bounded intervals for polytropic full gas dynamics in some parameter regimes. Definitely, the example of unstable shock waves on the whole line in \cite{zumbrun_convex_entropy} together with the asymptotic analysis in \cite{sandstede_abs_conv_instab,zumbrun_standingshock} of spectra in the whole-line limit 
shows that unstable steady solutions are possible 
on an interval for full gas dynamics with an artificial equation of state satisfying all of the usual requirements imposed in 
standard theory, including existence of a convex entropy, genuine nonlinearity of acoustic modes, etc.

\medskip

Moreover, due to the presence of spectral gap/absence of essential spectra in the bounded-interval problem, differently from
the whole- and half-line problems, changes in stability of the type considered in \cite{Serre_Zumbrun, zumbrun_standingshock},
involving passage of a real eigenvalue through zero, are associated necessarily with bifurcation/nonuniqueness,
by Lyapunov-Schmidt or center manifold reduction to the finite-dimensional case.\footnote{
See in particular the center manifold theory for generators of $C^0$ semigroups in \cite{haragus_Iooss} or the still more general
Fredholm-based Lyapunov-Shmidt reduction of \cite[Appendix D]{Monteiro2014} for closed densely defined operators with
an isolated crossing eigenvalue, along with the general finite-dimensional bifurcation result of
\cite[Lemma 3.10]{Barker_convexEntropy}. This is to be contrasted with the case of the whole line
discussed in \cite[\S 6.2]{Zumbrun_multi_dim_stab_planar}, for which $\lambda=0$ is embedded in essential spectrum and a crossing eigenvalue at
$\lambda=0$ may signal {\it either} steady state bifurcation as here or more complicated time-dependent bifurcations
involving far-field behavior and solutions of an
associated inviscid Riemann problem.}
Thus, any such violations of stability should yield also examples of large-amplitude nonuniqueness at the same time.
Small-amplitude uniqueness, on the other hand, follows readily by uniqueness of constant solutions, as follows by energy estimates
like those here, plus continuity. The investigation of large-amplitude uniqueness and stability for larger systems thus appears to be a very interesting direction
for future exploration; likewise, the study of the corresponding multi-D problem, for which existence/uniqueness of small-amplitude
solutions has been studied for example in \cite{kellogg_inflow_unbounded,kellogg_inflow_bounded,poiseuille_Stab_3d}.
In both 1- and multi-D, a very interesting open problem would be to study the asymptotic structure of solutions in
the small-viscosity limit, particularly in the multi-D case analogous to Poiseuille flow.

\subsection*{ Notation }

\noindent In this paper, C($\cdot$) denotes a nondecreasing and positive function and $C$ a generic notation whose exact values are of no importance.
$\lver \; \rver_{2}$ refers to the $L^{2}$-norm on $(0,1)$ and  $\lver \; \rver_{H^{n}}$, for $n\geq1$, to the $H^{n}$-norm. 
$\lver \; \rver_{\infty}$ refers to the $L^{\infty}$-norm on $[0,1]$.

\subsection*{ Acknowledgment } 

We would like to thank the anonymous referees for careful reading and for very valuable comments on the manuscript.

\section{Existence and uniqueness of steady states}\label{section_existence_steady}

\subsection{Analytical results}

In this part, we prove the following result.

\begin{prop}\label{existence_steady_sol}
Assume that $P$ is a smooth function. For any $(\rho_0, u_0, u_1)$, with $\rho_0>0,$ $u_0>0$ and $u_1>0$, problem \eqref{NS}-\eqref{BC} has a unique steady solution $\left(\hat{\rho},\hat{u} \right)$  with $\hat{\rho} > 0$.
\end{prop}

\begin{proof}
A steady solution $\left(\hat{\rho},\hat{u} \right)$ of \eqref{NS}-\eqref{BC} satisfies
\begin{equation*}
\left\{
\begin{aligned}
&\left(\hat{\rho} \hat{u} \right)_{x} = 0,\\
&\left(\hat{\rho} \hat{u}^{2} + P(\hat{\rho}) \right)_{x} = \nu \hat{u}_{xx},\\
&\hat{\rho}(0)=\rho_{0},\\
&\hat{u}(0)=u_{0} \text{  ,  } \hat{u}(1)=u_{1}.
\end{aligned}
\right.
\end{equation*}
Thus, $\hat{\rho} \hat{u}=\rho_{0} u _{0}$ and 
\begin{equation}\label{ode1}
\left\{
\begin{aligned}
&\nu \rho_{0} u_{0} \hat{\rho}_{x} = b \hat{\rho}^{2} - (\rho_{0} u_{0})^{2} \hat{\rho} - \hat{\rho}^{2} P(\hat{\rho}),\\
&\hat{\rho}(0)= \rho_{0}
\end{aligned}
\right.
\end{equation}
where $b$ is a constant that has to be determined. We define the map
\begin{equation*}
\phi := b \shortrightarrow \hat{\rho}(1) - \frac{\rho_{0} u_{0}}{u_{1}}
\end{equation*}
where $\hat{\rho}$ is the unique solution of System \eqref{ode1}. Notice that we only define $\phi$ when $\hat{\rho}$ is defined on $[0,1]$. Then, we remark that $\phi(\rho_{0} u_{0}^{2} + P(\rho_{0}))=\rho_{0} - \frac{\rho_{0} u_{0}}{u_{1}}$ and that $\phi$ is increasing. We also remark that if $b_{1}$ is in the domain of $\phi$, any $b<b_{1}$ is in the domain of $\phi$. Therefore, the domain of $\phi$ is an interval containing $\rho_{0} u_{0}^{2} + P(\rho_{0})$. Furthermore, $\phi$ is not bounded from above. Otherwise, one can show that $\hat{\rho}$ is bounded uniformly with respect to $b$ and then that $\phi(b)>b$ for $b$ large enough. One can also show that $\underset{b \shortrightarrow -\infty}{\lim} \phi(b) = - \frac{\rho_{0} u_{0}}{u_{1}}$. Therefore, there exists a unique $b$ such that $\hat{\rho}(1)=\frac{\rho_{0} u_{0}}{u_{1}}$ and $\hat{\rho}>0$.
\end{proof}

\noindent Notice that a solution of \eqref{NS} is constant if and only if  $u_{0} = u_{1}$.

\medskip

In the following, we study among other things the stability of almost constant steady solutions of \eqref{NS}, by which we mean
solutions satisfying
\begin{equation}\label{almost_constant_assump}
\exists \epsilon > 0 \text{  ,  } \varepsilon \ll 1 \text{   and   } \left|u_{0} - u_{1} \right| \leq \epsilon.
\end{equation}
For this, the following lemma will be useful.
\begin{lemma}\label{control_almost_constant_state}
Assume that we are under the assumptions of Proposition \ref{existence_steady_sol}. Then, the unique steady solution $(\hat{\rho}, \hat{u})$ of problem \eqref{NS}-\eqref{BC} is smooth and if $u_{1} \neq u_{0}$ we have
\begin{equation*}
\hat{\rho} > 0 \text{  ,  } \hat{u} > 0, \left( u_{1} - u_{0} \right) \hat{\rho}_{x} < 0 \text{  ,  } \left( u_{1} - u_{0} \right) \hat{u}_{x} > 0, \underset{u_{1} \shortrightarrow u_{0}}{\lim}\left( \lver \hat{\rho}_{x} \rver_{\infty} + \lver \hat{u}_{x} \rver_{\infty} \right) = 0.
\end{equation*}
\end{lemma}

\begin{proof}
The first four inequalities are clear (notice that if $u_{0} \neq u_{1}$, $|\hat{u}_{x}|>0$). The last inequality follows from a comparison argument and the continuity of the map $\phi$.
\end{proof}

We denote solutions as {\it compressive} when $\hat{u}_{x} > 0$ and {\it expansive} when $\hat{u}_{x} < 0$.

\subsection{Numerical simulations}
A steady solution $\left(\hat{\rho},\hat{u} \right)$ of \eqref{NS}-\eqref{BC} is characterized by system \eqref{ode1} where $b$ is the unique zero of $\phi$. The numerical computation of such a solution is carried out in two steps:

\medskip

\noindent - We compute $b$ with a Newton's method. We initiate the process with $b=b_{0}$ where

\begin{equation*}
b_{0} = \rho_{0} u_{0}^{2} + P(\rho_{0}).
\end{equation*}
Note that for a small viscosity ($\nu \leq 1$), the initial starting point $b=b_{0}$ ceases to be relevant. Thus, in this case, we use a dichotomy method to find a better starting point.

\medskip

\noindent - The solution of system \eqref{ode1} is computed with a four-order Runge-Kutta method.

\medskip
\medskip

We display the results of numerical simulations for a monatomic pressure law $P(\rho) = \rho^{1.4}$ with
$\nu=1$. Figure \ref{steadysol1} represents the expansive solution for $u_{0}=2$, $u_{1}=3$ and $\rho_{0}=3$. Figure \ref{steadysol2} represents the compressive solution when $u_{0}=1.5$, $u_{1}=1$ and $\rho_{0}=2$.

\begin{figure}[!t]
\centering
\includegraphics[scale=0.25]{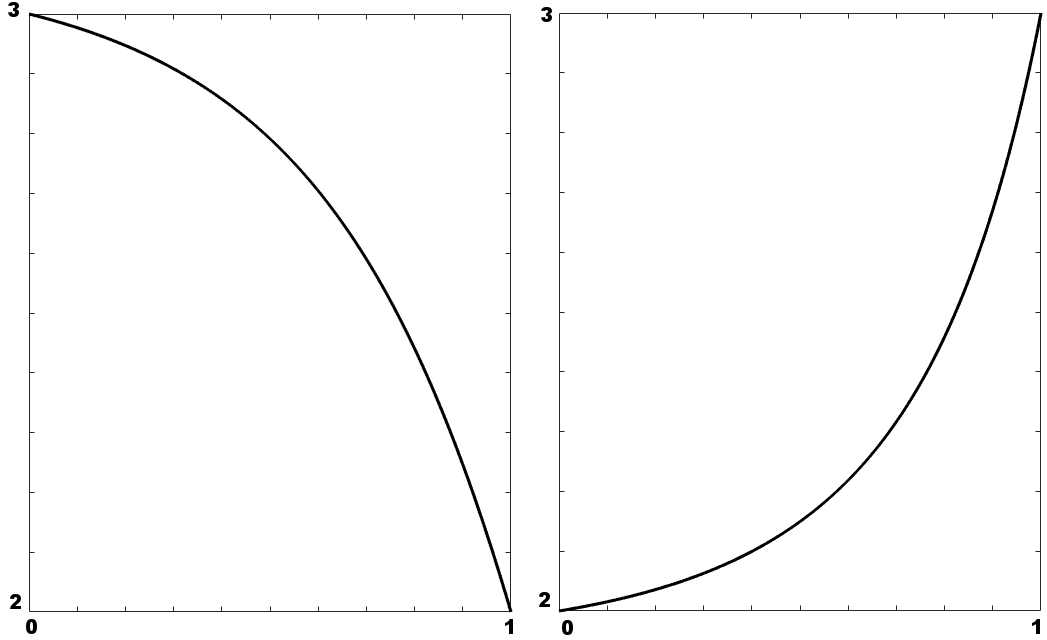}
\caption{A steady expansive solution of \eqref{NS}. Left: $\rho$ ; Right: $u$}
\label{steadysol1}
\end{figure}

\begin{figure}[!t]
\centering
\includegraphics[scale=0.25]{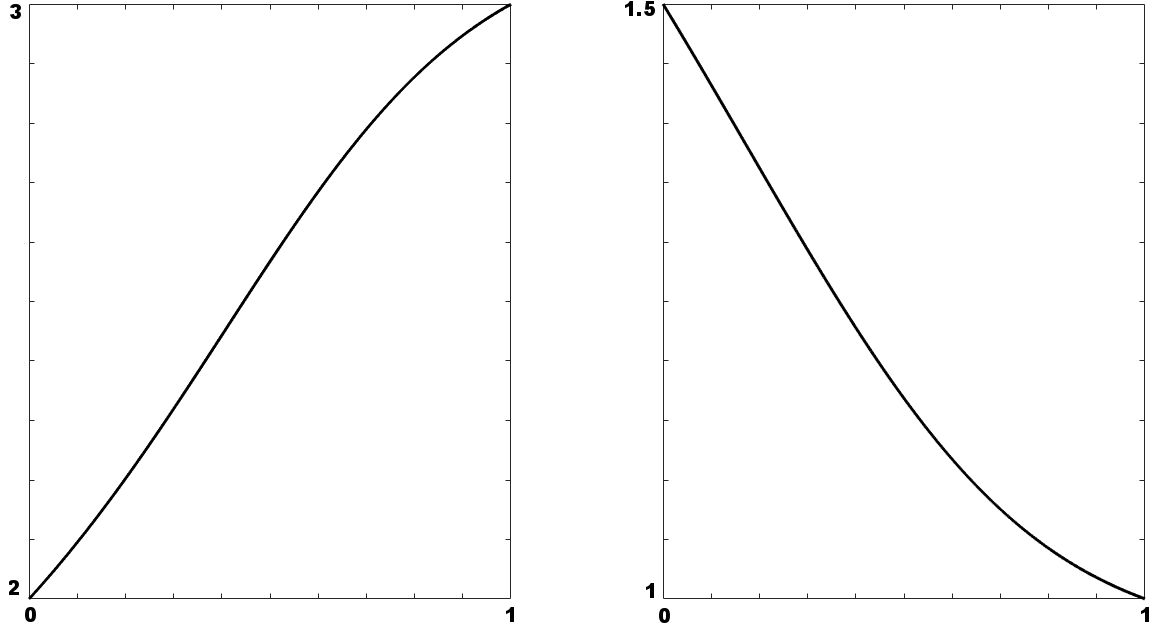}
\caption{A steady compressive solution of \eqref{NS}. Left: $\rho$ ; Right: $u$}
\label{steadysol2}
\end{figure}

\section{Linear estimates}\label{section_linear_estimate}

\subsection{The eigenvalue problem}\label{eigenvalue_problem_part}
In order to study the stability of steady states, we linearize system \eqref{NS} about the steady state $\left(\hat{\rho},\hat{u} \right)$. Then, we study the corresponding eigenvalue problem for $(r,v)$
\begin{equation}\label{linear_steady}
\left\{
\begin{aligned}
&\lambda r + \left( \hat{\rho} v +\hat{u} r \right)_{x} = 0,\\
&\lambda \hat{\rho} v + \left(\hat{\rho} \hat{u} v + P'(\hat{\rho}) r \right)_{x} + \hat{u}_{x} \left(\hat{u} r + \hat{\rho} v \right) = \nu v_{xx},
\end{aligned}
\right.
\end{equation}
with
\begin{equation}\label{BC_linear_steady}
r(0)=v(0)=v(1) = 0.
\end{equation}
We define the linear unbounded operator
\begin{equation}\label{linear_op}
\mathcal{L} \left( r , v \right) = \begin{pmatrix} - \left( \hat{\rho} v +\hat{u} r  \right)_{x} \\ \nu v_{xx} - \left(\hat{\rho} \hat{u} v + P'(\hat{\rho}) r \right)_{x} - \hat{u}_{x} \left(\hat{u} r + \hat{\rho} v \right) \end{pmatrix},
\end{equation}
for $(r,v)$ in the domain $\mathcal{D}(\mathcal{L}) = \left\{(r,v) \in H^{1} \times H^{2} \text{, } r(0)=v(0)=v(1)= 0 \right\}$ and the matrix
\begin{equation}\label{linear_op_S}
\mathcal{S} = \begin{pmatrix} 1 & 0 \\ 0 &  \hat{\rho} \end{pmatrix}.
\end{equation}

\begin{remark}
For constant steady states, the eigenvalue problem simplifies into
\begin{equation}\label{linear_constant}
\left\{
\begin{aligned}
&\lambda r + \hat{\rho} v_{x} +\hat{u} r_{x} = 0,\\
&\lambda \hat{\rho} v +  \hat{\rho} \hat{u} v_{x} + P'(\hat{\rho}) r_{x} = \nu v_{xx}.
\end{aligned}
\right.
\end{equation}
\end{remark}
In the following, we denote by $\sigma(\mathcal{S}^{-1} \mathcal{L})$ the spectrum of $\left( \mathcal{S}^{-1} \mathcal{L}, \mathcal{D}(\mathcal{L}) \right)$ in $L^{2}(0,1)$. The following proposition shows that $\sigma(\mathcal{S}^{-1} \mathcal{L})$ only contains eigenvalues.
\begin{prop}\label{compact_inverse_L}
The inverse of $\mathcal{L}$ exists and is compact and $\sigma(\mathcal{S}^{-1} \mathcal{L})$ only contains eigenvalues. Furthermore, the spectrum of $\left( \mathcal{S}^{-1} \mathcal{L}, \mathcal{D}(\mathcal{L}) \cap H^{2} \times H^{3} \right)$ in the space $\left\{ (r,v) \in H^{1} \text{, } r(0)=v(0)=v(1)= 0 \right\}$ only contains eigenvalues.
\end{prop}

\begin{proof}
First we show that $0 \notin \sigma(\mathcal{L})$. For $f$ and $g$ in $L^{2}(0,1)$, we solve
\begin{equation*}
\mathcal{L} (\rho,v) = (f,g) \text{  with  } r(0)=v(0)=v(1) = 0.
\end{equation*}
This leads to the system
\begin{equation*}
\left\{
\begin{aligned}
&\hat{\rho} v +\hat{u} r = -\int_{0}^{x} f(y),\\
&\nu v_{x} = \nu v_{x}(0) + \hat{\rho} \hat{u} v - \frac{P'(\hat{\rho})}{\hat{u}} \left(\hat{\rho} v + \int_{0}^{x} f \right) + \int_{0}^{x} g(y) - \int_{0}^{x} \hat{u}_{x} \int_{0}^{y} f(y).
\end{aligned}
\right.
\end{equation*}

\noindent Then, we can solve the second equation with the initial condition $v(0)=0$
\small
\begin{equation*}
v(x) = \frac{1}{\nu} \int_{0}^{x} \exp \left( \frac{1}{\nu} \int_{y}^{x}  \hat{\rho} \hat{u} - \frac{P'(\hat{\rho})}{\hat{u}} \hat{\rho} dz \right) \left( \nu v_{x}(0) - \frac{P'(\hat{\rho})}{\hat{u}} \int_{0}^{y} f  + \int_{0}^{y} g - \int_{0}^{y} \hat{u}_{x} \int_{0}^{z} f \right) dy.
\end{equation*}
\normalsize

\noindent Since $v(1)=0$, we can compute $v_{x}(0)$ and it implies that $\mathcal{L}$ is invertible. Furthermore, if $f$ and $g$ are bounded in $L^{2}(0,1)$, we get from the previous equality that $v_{x}(0)$ is bounded and that $v$ is bounded in $H^{1}(0,1)$. The first statement follows easily. The second statement follows from similar computations.
\end{proof}

\begin{remark}
\noindent Notice that $(\hat{\rho}',\hat{u}')$ can not be an eigenfunction of $\mathcal{S}^{-1} \mathcal{L}$ for the eigenvalue $\lambda=0$ since it does not satisfy the boundary conditions \eqref{BC_linear_steady}. This differs from the whole line case.
\end{remark}

In order to prove the spectral stability of steady solutions, we need high frequency estimates for problem \eqref{linear_steady}-\eqref{BC_linear_steady}. First, we establish a useful lemma.

\begin{lemma}\label{high_frequency_control}
For any $(r,v)$ satisfying the boundary conditions \eqref{BC_linear_steady} and $\lver \lambda \rver$ large enough, we have
\begin{align*}
&\lver \tilde{r} \rver_{2}^{2} \leq \frac{C}{\lver \lambda \rver} \left( \lver r \rver_{2}^{2} + \lver v \rver_{2}^{2} + \lver \left( \lambda \mathcal{S} - \mathcal{L} \right)(r,v)_{1} \rver^{2}_{2} \right),\\
&\lver r \rver_{2}^{2} + \lver v \rver_{2}^{2} \leq \frac{C}{\lver \lambda \rver} \left( \lver r_{x} \rver_{2}^{2} + \lver v_{x} \rver_{2}^{2} + \lver \left( \lambda \mathcal{S} - \mathcal{L} \right)(r,v) \rver^{2}_{2} \right),\\
&\lver v_{x} \rver_{2}^{2} \leq \frac{C}{\lver \lambda \rver} \left( \lver r_{x} \rver_{2}^{2} + \lver v_{xx} \rver_{2}^{2} + \lver \left( \lambda \mathcal{S} - \mathcal{L} \right) (r,v) \rver^{2}_{H^{1}} \right),
\end{align*}
where $\tilde{r}(x) = \int_{0}^{x} r(y) dy$.
\end{lemma}

\begin{proof}

If we denote $\left( \lambda \mathcal{S} - \mathcal{L} \right) (r,v) = (f,g)$ and $\tilde{f}(x) = \int_{0}^{x} f(y) dy$, we have 
\begin{equation}\label{system_int}
\left\{
\begin{aligned}
&\lambda \tilde{r} + \hat{\rho} v +\hat{u} r = \tilde{f},\\
&\lambda r + \left( \hat{\rho} v +\hat{u} r \right)_{x} = f,\\
&\lambda \hat{\rho} v + \left(\hat{\rho} \hat{u} v + P'(\hat{\rho}) r \right)_{x} + \hat{u}_{x} \left(\hat{u} r + \hat{\rho} v \right) = \nu v_{xx} + g.
\end{aligned}
\right.
\end{equation}
Thus, we easily see that
\begin{equation*}
\lver \tilde{r} \rver_{2}^{2} \leq \frac{C}{\lver \lambda \rver} \left( \lver r \rver^{2}_{2} + \lver v \rver^{2}_{2} + \lver f \rver^{2}_{2} \right).
\end{equation*}
Furthermore, by integrating by parts, we get
\small
\begin{equation*}
\lver r \rver_{2}^{2} + \lver \sqrt{\hat{\rho}} v \rver_{2}^{2} = \frac{1}{\lver \lambda \rver} \lver \int_{0}^{1} \overline{r} \lambda r \rver +  \frac{1}{\lver \lambda \rver} \lver \int_{0}^{1} \overline{v} \lambda \hat{\rho} v \rver \leq \frac{C}{\lver \lambda \rver} \left( \lver r \rver^{2}_{2} + \lver \left[ \hat{u} \lver r \rver^{2} \right]_{0}^{1} \rver + \lver v \rver_{H^{1}}^{2} + \lver (f,g) \rver^{2}_{2} \right)
\end{equation*}
\normalsize 
and the second inequality follows from Lemma \ref{Linfty_controls}. Finally, by differentiating the second equation of \eqref{system_int}, we obtain
\small
\begin{equation*}
\lver \left(\hat{\rho} v \right)_{x} \rver_{2}^{2} = \frac{1}{\lver \lambda \rver} \lver \int_{0}^{1} \overline{\left(\hat{\rho} v \right)_{x}} \lambda \left(\hat{\rho} v \right)_{x} \rver \leq \frac{C}{\lver \lambda \rver} \left( \lver r \rver_{H^{1}}^{2} + \lver v \rver_{H^{2}}^{2} + \lver \left[\hat{\rho} \overline{v_{x}} \left(\nu v_{xx} - P'(\hat{\rho}) r_{x} \right) \right]_{0}^{1} \rver + \lver (f,g) \rver^{2}_{H^{1}} \right).
\end{equation*} 
\normalsize
Then, we notice that
\small
\begin{equation*}
\left\{
\begin{aligned}
&\nu v_{xx}(0) - P'(\rho_{0}) r_{x}(0) = \rho_{0} u_{0} v_{x}(0)- g(0),\\
&\nu v_{xx}(1) - P'(\hat{\rho}(1)) r_{x}(1) = \hat{\rho}(1) u_{1} v_{x}(1) + P''(\hat{\rho}(1)) \hat{\rho}_{x}(1) r(1) + \hat{u}_{x}(1) u_{1} r(1) - g(1)
\end{aligned}
\right.
\end{equation*}
\normalsize
and thanks to Lemma \ref{Linfty_controls},  we get
\begin{equation*}
\lver \left[\hat{\rho} \overline{v_{x}} \left(\nu v_{xx} - P'(\hat{\rho}) r_{x} \right) \right]_{0}^{1} \rver \leq C \lver v \rver_{H^{2}}^{2} + C \lver r \rver_{H^{1}}^{2} + C \lver g \rver_{H^{1}}^{2}.
\end{equation*}
The result follows easily.
\end{proof}

We can now establish a high frequency estimate in $H^{1}$. 

\begin{prop}\label{high_freq_estimate}
Assume that $P$ satisfies \eqref{pressure_cond}. There exists a constant $\alpha > 0$ such that if $\Re(\lambda) > - \alpha$ and $\lver \lambda \rver$ is large enough,
\begin{equation*}
\lver (r,v) \rver^{2}_{H^{1}} \leq C \lver \left(\lambda \mathcal{S} - \mathcal{L} \right) (r,v) \rver^{2}_{H^{1} \times L^2} 
\end{equation*}
for any $(r,v)\in \left\{(r,v) \in H^{1} \times H^{1} \text{, } r(0)=v(0)=v(1)= 0 \right\}$. 
\end{prop}

\begin{proof}
This proof is based on an appropriate Goodman-type energy estimate. In the following we denote $ \left(\lambda \mathcal{S} - \mathcal{L}\right) (\rho,v) =(f,g)$. We define the energy 
\begin{equation*}
\mathcal{E} \left(r,v \right) = \frac{1}{2} \int_{0}^{1} \phi_{1} \lver r_{x} \rver^{2} + \phi_{2} \lver \left( \hat{\rho} v \right)_{x} \rver^{2}
\end{equation*}
where $\phi_{1}$ and $\phi_{2}$ satisfy
\begin{equation*}
\phi_{1} > 0 \text{ ,  } \phi_{2} > 0 \text{ ,  }  \phi_{1} = P'(\hat{\rho}) \phi_{2} \text{ ,  } \frac{1}{2} (\hat{u} \phi_{1})_{x} - 2 \hat{u}_{x} \phi_{1} < 0\footnote{For instance, $\phi_{1}(0)=1$ and $\hat{u} \phi_{1}'=3 \hat{u}' \phi_{1} - \delta \hat{u}$ for $\delta>0$ small enough.}.
\end{equation*}
This energy is equivalent to the $H^{1}$-norm by the Poincar\'e inequality (see Lemma \ref{poincare}). Then, we compute 
\begin{equation*}
2 \Re(\lambda) \mathcal{E} \left(r,v \right) = \Re \left( \int_{0}^{1} \phi_{1} \overline{r_{x}} \lambda r_{x} \right) + \Re \left( \int_{0}^{1} \phi_{2} \overline{\left( \hat{\rho} v \right)_{x}} \lambda \left( \hat{\rho} v \right)_{x} \right).
\end{equation*}
Arguing by density, we assume that $(r,v) \in H^{2} \times H^{3}$. We have
\begin{equation*}
\begin{aligned}
2 \Re(\lambda) \mathcal{E} \left(r,v \right) \leq &-  \int_{0}^{1} \Re \left[ \phi_{1} \overline{r_{x}} \left( \hat{u} r_{xx} + \hat{\rho} v_{xx} + \hat{u}_{x} r_{x} \right] \right) + C \lver r_{x} \rver_{2} \left( \lver r \rver_{2} + \lver v \rver_{H^{1}} + \lver f \rver_{H^{1}} \right)\\ 
&+\int_{0}^{1} \Re \left[ \phi_{2} \hat{\rho} \overline{v_{x}} \left( \nu v_{xxx} - P'(\hat{\rho}) r_{xx} - \hat{\rho} \hat{u} v_{xx} \right) \right] + \int_{0}^{1} \phi_{2} \hat{\rho} \overline{v_{x}} g_{x}\\
&+\int_{0}^{1} \Re \left[ \phi_{2} \hat{\rho}_{x} \overline{v} \left( \nu v_{xxx} - P'(\hat{\rho}) r_{xx} \right] \right) + C \lver v \rver_{H^{1}} \left( \lver r \rver_{H^{1}} + \lver v \rver_{H^{1}} + \lver v_{xx} \rver_{2}\right).
\end{aligned}
\end{equation*}
Therefore, we have
\small
\begin{equation*}
\begin{aligned}
2 \Re(\lambda) \mathcal{E} \left(r,v \right) \leq & \int_{0}^{1} -\nu \hat{\rho} \phi_{2} \lver v_{xx} \rver^{2} + \left(\frac{1}{2} (\hat{u} \phi_{1})_{x} - 2 \hat{u}_{x} \phi_{1} \right) \lver r_{x} \rver^{2} + \Re(\overline{r_{x}} v_{xx}) \left(- \hat{\rho} \phi_{1} + \hat{\rho} P'(\hat{\rho}) \phi_{2} \right)\\
&+ \Re \left( \left[-\frac{1}{2} \hat{u} \phi_{1} \lver r_{x} \rver^{2} + \phi_{2} \hat{\rho} \overline{v_{x}} \left(\nu v_{xx} - P'(\hat{\rho}) r_{x} + g \right) \right]_{0}^{1} \right)\\
&+ C \lver r_{x} \rver_{2} \left( \lver r \rver_{2} + \lver v \rver_{H^{1}} + \lver f \rver_{H^{1}} \right) + C \lver v \rver_{H^{1}} \left( \lver r \rver_{H^{1}} + \lver v \rver_{H^{1}} + \lver v_{xx} \rver_{2}\right) + C \lver g \rver_{2} \lver v_{xx} \rver_{2}.
\end{aligned}
\end{equation*}
\normalsize
Then, we notice that
\begin{equation}\label{boundary_equalities}
\begin{aligned}
& u_{0} r_{x}(0) = - \rho_{0} v_{x}(0) + f(0),\\
&\nu v_{xx}(0) - P'(\rho_{0}) r_{x}(0) + g(0) = \rho_{0} u_{0} v_{x}(0),\\
&\nu v_{xx}(1) - P'(\hat{\rho}(1)) r_{x}(1) + g(1) = \hat{\rho}(1) u_{1} v_{x}(1) + \left( P''(\hat{\rho}(1)) \hat{\rho}_{x}(1) + \hat{u}_{x}(1) u_{1} \right)  r(1)
\end{aligned}
\end{equation}
and thanks to Lemma \ref{Linfty_controls}, we obtain
\small
\begin{equation*}
\Re \left( \left[-\frac{1}{2} \hat{u} \phi_{1} \lver r_{x} \rver^{2} + \phi_{2} \hat{\rho} \overline{v_{x}} \left(\nu v_{xx} - P'(\hat{\rho}) r_{x} + g \right) \right]_{0}^{1} \right) \leq C \lver r \rver_{2} \lver r_{x} \rver_{2} + C \lver v_{x} \rver_{2} \lver v_{x} \rver_{H^{1}} + C \lver f \rver_{H^{1}}^{2}.
\end{equation*}
\normalsize
Then, using the second and the third inequality of Lemma \ref{high_frequency_control} and Young's inequality, we can find a constant $\alpha > 0$, such that for $\lver \lambda \rver$ large enough,
\begin{equation*}
2 \Re(\lambda) \mathcal{E} \left(r,v \right) \leq - \alpha \lver (r_{x},v_{xx}) \rver_{2}^{2} + C \lver \left(\lambda \mathcal{S} - \mathcal{L}\right) (r,v) \rver_{H^{1} \times L^{2}}^{2}.
\end{equation*}
Since $\mathcal{E}$ is a norm equivalent to the $H^{1}$-norm, the inequality follows from the Poincar\'e-Wirtinger inequality on $v_{x}$.
\end{proof}

\subsection{The linear time evolution problem}
In this part, we study the linearization of system \eqref{NS} about the steady state $\left(\hat{\rho},\hat{u} \right)$
\begin{equation*}
\left\{
\begin{aligned}
&\mathcal{S} \begin{pmatrix} r \\ v \end{pmatrix}_{t} - \mathcal{L} \begin{pmatrix} r \\ v \end{pmatrix} = 0,\\
&\left( r,v \right)_{|t=0} = \left( r_{0},v_{0} \right),\\
&r(0)=v(0)=v(1)=0.
\end{aligned}
\right.
\end{equation*}
We define for $k \in \mathbb{N}$ the spaces
\begin{align*}
&\mathcal{H}_{k} = \left\{ (r,v) \in H^{k} \text{, } \frac{d^{l}r}{dx^{l}}(0)=\frac{d^{l}v}{dx^{l}}(0)=\frac{d^{l}v}{dx^{l}}(1)= 0 \text{ , for any } l<k \right\}\\
&\mathcal{H}_{1,0} = \left\{(r,v) \in H^{1} \times L^{2} \text{, } r(0)=0 \right\}.
\end{align*}
The main goal of this subsection is to show a linear exponential stability in $\mathcal{H}_{1,0}$. This will help us to show the nonlinear exponential stability (see Remark \ref{H1XL2_explanation}).
\noindent The following lemmas show that $\mathcal{S}^{-1} \mathcal{L}$ generates a $\mathcal{C}^{0}$-semigroup on $\mathcal{H}_{k}$ and $\mathcal{H}_{1,0}$.  

\begin{lemma}\label{gen_C0_semigroup_L2}
The operator $\left( \mathcal{S}^{-1} \mathcal{L}, \mathcal{D}(\mathcal{L}) \right)$ is closed densely defined on $L^{2}$ and generates a $\mathcal{C}^{0}$-semigroup. Similarly $\left( \mathcal{S}^{-1} \mathcal{L},  \mathcal{H}_{k} \cap H^{k+1} \times H^{k+2} \right)$ is closed densely defined on $\mathcal{H}_{k}$ and generates a $\mathcal{C}^{0}$-semigroup.
\end{lemma}

\begin{proof}
The proof is similar to the proof of Proposition 2.2 in \cite{mascia_zumbrun}. In the following we denote $(f,g)=(\lambda-\mathcal{S}^{-1}\mathcal{L})(r,v)$. For $\lambda>0$ and $(r,v) \in  \mathcal{D}(\mathcal{L})$
\begin{equation*}
\begin{aligned}
\lambda \lver r \rver_{2}^{2} + \left( \lambda \hat{\rho} v , \frac{1}{\hat{\rho}} v \right)_{2} &= -\left( \hat{\rho} r_{x}, r \right)_{2} - \left( \hat{u} v_{x}, r \right)_{2} + \left( f, r \right)_{2}  + \nu \left(v_{xx}, v \right)_{2}\\
&\hspace{0.45cm} -\left( P'(\hat{\rho}) r_{x}, v \right)_{2} - \left( \hat{\rho} \hat{u} v_{x}, r \right)_{2} + \left( g, v \right)_{2} + C \lver (r,v) \rver_{2}^{2} \\
&\leq - \nu \lver v_{x} \rver_{2}^{2} + \left( f, r \right)_{2}  + \left( g, v \right)_{2} + C \lver (r,v) \rver_{2} \left( \lver (r,v) \rver_{2} + \lver v_{x} \rver_{2} \right)
\end{aligned}
\end{equation*}
where we have integrated by parts and we have noticed a good sign for $\lver r(1) \rver^{2}$. Applying Young's inequality, there exists a constant $C_{0}>0$ such that 
\begin{equation}\label{o}
\lambda \lver r \rver_{2}^{2} + \lambda \lver v \rver_{2}^{2} \leq C_{0} \lver (r,v) \rver_{2}^{2}  +\lver (r,v) \rver_{2} \lver (f,g) \rver_{2}.
\end{equation}
Dividing by $ \lver (r, v) \rver_{2}$, we get
\begin{equation*}
\lambda \lver (r, v) \rver_{2} \leq C_{0} \lver (r,v) \rver_{2} +  \lver (\lambda-\mathcal{S}^{-1} \mathcal{L})(r,v) \rver_{2},
\end{equation*}
hence for $\lambda > C_{0}$
\begin{equation}\label{int1_C0semigroup}
\lver (r, v) \rver_{2} \leq \frac{1}{\lambda - C_{0}} \lver (\lambda-\mathcal{S}^{-1}\mathcal{L})(r,v) \rver_{2}.
\end{equation} 
Similarly, we have for $(r,v) \in \mathcal{H}_{1} \cap H^{2} \times H^{3}$
\small
\begin{equation*}
\lambda \lver r_{x} \rver_{2}^{2} + \lambda \lver v_{x} \rver_{2}^{2} = (f_{x},r_{x})_{2} + (g_{x},v_{x})_{2} -(\hat{u} r_{xx},r_{x})_{2} + (v_{xxx}, \frac{\nu}{\hat{\rho}} v_{x})_{2} + C \left( \lver v_{xx} \rver_{2} + \lver (r_{x},v_{x}) \rver_{2} \right) \lver(r,v) \rver_{H^{1}} 
\end{equation*}
\normalsize
and there exists a constant $C_{1}>0$ such that for any $\lambda>0$
\begin{equation}\label{t}
\lambda \lver (r_{x}, v_{x}) \rver_{2}^{2} \leq C_{1} \lver (r,v) \rver_{H^{1}}^{2} + \lver (r_{x},v_{x}) \rver_{2} \lver (f_{x},g_{x}) \rver_{2}.
\end{equation}
\noindent Summing \eqref{o} and \eqref{t}, and noting that
\begin{equation*}
\lver (r,v) \rver_{2} \lver (f,g) \rver_{2} + \lver (r_{x},v_{x}) \rver_{2} \lver (f_{x},g_{x}) \rver_{2}
\leq \lver(r,v)\rver_{H^1} \lver(f,g)\rver_{H^1}
\end{equation*}
by Cauchy-Schwarz' inequality, we get
\begin{equation*}
\lambda \lver (r, v) \rver_{H^{1}}^{2} \leq (C_{0}+C_{1}) \lver (r,v) \rver_{H^{1}}^{2} + \lver (r,v) \rver_{H^{1}} \lver (\lambda-\mathcal{S}^{-1} \mathcal{L})(r,v) \rver_{H^{1}}
\end{equation*}
hence for $\lambda > C_{0}+C_{1}$,
\begin{equation}\label{int2_C0semigroup}
\lver (r, v) \rver_{H^{1}} \leq \frac{1}{\lambda - C_{0}-C_{1}} \lver (\lambda-\mathcal{S}^{-1}\mathcal{L})(r,v) \rver_{H^{1}}.
\end{equation}

Since we know from Proposition \ref{compact_inverse_L} that the spectrum of $\mathcal{S}^{-1} \mathcal{L}$ only contains eigenvalues, the inequalities \eqref{int1_C0semigroup} and \eqref{int2_C0semigroup} give resolvent bounds and it shows that $\mathcal{S}^{-1} \mathcal{L}$ generates a $\mathcal{C}^{0}$-semigroup on $L^{2}$ and $\mathcal{H}_{1}$ by the Hille-Yosida theorem (see also \cite{pazy_semigroup}). The case $k \geq 2$ is a small adaptation of the previous estimates.
\end{proof}

In the following, we denote this $\mathcal{C}^{0}$-semigroup by $e^{t \mathcal{S}^{-1} \mathcal{L}}$.

\begin{lemma}
There exists a constant $\omega>0$ such that for any $(r_{0},v_{0}) \in \mathcal{H}_{1,0}$
\begin{equation}\label{eg}
\lvert e^{t\mathcal{S}^{-1} \mathcal{L}} (r_{0},v_{0}) \rvert_{H^{1} \times L^{2}} \leq e^{\omega t} \lvert (r_{0},v_{0}) \rvert_{H^{1} \times L^{2}}.
\end{equation}
Furthermore, $e^{t\mathcal{S}^{-1} \mathcal{L}}$ is a $\mathcal{C}^{0}$-semigroup on $\mathcal{H}_{1,0}$.
\end{lemma}

\begin{proof}
We argue by density and we take $(r_{0},v_{0}) \in \mathcal{H}_{2}$.  In the following, we denote $(r(t),v(t))=e^{t\mathcal{S}^{-1} \mathcal{L}} (r_{0},v_{0})$. Notice that we have $r(t,0)=v(t,0)=v(t,1)=r_{x}(t,0)=0$ for any $t\geq 0$.  We define the energy
\begin{equation*}
\mathcal{E}(r,v)= \frac{A}{2} \left( \lvert r \rvert_{2}^{2} + \left( \hat{\rho} v, v \right)_{2} \right) + \left(v, \hat{\rho}^{2} r_{x} \right)_{2} + \frac{\nu}{2} \lver r_{x} \rver_{2}^{2}.
\end{equation*}
In the following we take $A>0$ large enough. In particular, $\mathcal{E}$ is equivalent to the $H^{1} \times L^{2}$-norm. We get
\begin{equation*}
\begin{aligned}
\frac{d}{dt} \mathcal{E}(r,v) &\leq \nu A \left(v_{xx} , v \right)_{2}  + AC \lver (r,v) \rver_{2} \lver (r,v) \rver_{H^{1}} + \left( \nu v_{xx} , \hat{\rho} r_{x} \right)_{2} - \left( v , \hat{\rho}^{3} v_{xx} \right)_{2}\\
&\hspace{0.45cm} - \left( v , \hat{\rho}^{2} \hat{u} r_{xx} \right)_{2} - \left( \nu \hat{\rho} v_{xx} , r_{x} \right)_{2} - \nu \left(\hat{u} r_{xx}, r_{x} \right)_{2}+ C\left( \lver v \rver_{2} +  \lver r_{x} \rver_{2} \right)  \lver (r,v) \rver_{H^{1}}\\
&\leq - A \nu \lver v_{x} \rver_{2}^{2} + \left( v_{x} , \hat{\rho}^{3} v_{x} \right)_{2} + \frac{\nu}{2} \left(\hat{u}_{x} r_{x}, r_{x} \right)_{2} + C\left( A \lver (r,v) \rver_{2} + \lver r_{x} \rver_{2} \right) \lver (r,v) \rver_{H^{1}}\\
\end{aligned}
\end{equation*}
where  we have used cancellation of the highest-order terms $\pm \left( \nu v_{xx} , \hat{\rho} r_{x} \right)_{2} $, we have integrated by parts and we have noticed a good sign for $r_{x}(1)^{2}$. Then, applying Young's inequality, we obtain for some $\omega>0$ large enough,
\begin{equation*}
\frac{d}{dt} \mathcal{E}(r,v) \leq  2 \omega \left(  \lvert r \rvert_{2}^{2} + \lvert v \rvert_{2}^{2} + \lvert r_{x} \rvert_{2}^{2} \right) \leq  2\omega  \mathcal{E}.
\end{equation*}
The inequality \eqref{eg} follows easily. 
Finally for any $U_{0} \in \mathcal{H}_{1,0}$ and $V_{0} \in \mathcal{H}_{1}$
\begin{equation*}
\lver e^{t\mathcal{S}^{-1} \mathcal{L}} U_{0} - U_{0} \rver_{H^{1} \times L^{2}} \leq \lver e^{t\mathcal{S}^{-1} \mathcal{L}} V_{0} - V_{0} \rver_{H^{1}} + (1+e^{\omega t}) \lver U_{0}-V_{0} \rver_{H^{1} \times L^{2}}
\end{equation*}
and continuity at $t=0$ follows since $e^{t\mathcal{S}^{-1} \mathcal{L}}$ is continuous at $t=0$ on $\mathcal{H}_{1}$ and $\mathcal{H}_{1}$  is dense in $\mathcal{H}_{1,0}$.
\end{proof}

The following proposition gives linear exponential stability under the assumption of a spectral gap. It is the main result of this subsection.

\begin{prop}\label{pruss_thm}
Assume that $P$ satisfies \eqref{pressure_cond}. Assume that there exists a constant $\alpha>0$, such that $\Re \sigma(\mathcal{S}^{-1} \mathcal{L}) < - \alpha$. Then, there exists $\theta$ and $C$,   $0 < \theta < \alpha$, such that  for any $(r_{0},v_{0}) \in \mathcal{H}_{1,0}$
\begin{equation*}
\lver e^{t \mathcal{S}^{-1} \mathcal{L}}(r_{0},v_{0}) \rver_{H^{1} \times L^{2}} \leq C e^{-\theta t} \lver (r_{0},v_{0}) \rver_{H^{1} \times L^{2}}.
\end{equation*}
\end{prop}

\begin{proof}
If $(r_{0},v_{0}) \in \mathcal{H}_{1}$, Proposition \ref{high_freq_estimate} gives two constants $C$ and $\theta$ with $0 < \theta < \alpha$ such that for any $\lambda \in \mathbb{C}$ satisfying $\mathcal{R}(\lambda) = - \theta$
\begin{equation*}
\lver \left(\lambda - \mathcal{S}^{-1} \mathcal{L} \right)^{-1} (r_{0},v_{0}) \rver_{H^{1} \times L^{2}} \leq C \lver (r_{0},v_{0}) \rver_{H^{1} \times L^{2}}.
\end{equation*}
\noindent The result follows by density and Pr\"uss' theorem (see for instance \cite{Pruss_theorem,Yosida_book}).

\end{proof}

\section{Spectral stability}\label{section_spectral_stab}

\subsection{Constant and almost constant states}

First, we study the spectral stability of constant states.

\begin{prop}\label{spec_stab_constant}
Assume that $\left(\hat{\rho},\hat{u} \right)$ is a constant solution of \eqref{NS} and that $P$ satisfies \eqref{pressure_cond}. Then, there exists $\alpha>0$, $\Re \sigma(\mathcal{S}^{-1} \mathcal{L}) \leq - \alpha$.
\end{prop}

\begin{proof}
Computing $\Re \left( \left(\eqref{linear_constant}_{1}, P'(\hat{\rho}) \overline{r} \right)_{L^{2}} + \left(\eqref{linear_constant}_{2}, \hat{\rho} \overline{v} \right)_{L^{2}}\right)$, we get 
\begin{equation*}
\Re(\lambda) \left( \lver \sqrt{P'(\hat{\rho})} r \rver^{2}_{2} + \lver \sqrt{\hat{\rho}} v \rver^{2}_{2} \right) + \nu \lver v_{x} \rver_{2}^{2} + \frac{1}{2} P'(\hat{\rho}(1)) u_{1} \lver r(1) \rver^{2} = 0.
\end{equation*}
Thus, $\Re(\lambda) < 0$. The result follows from Proposition \ref{compact_inverse_L} and Proposition \ref{high_freq_estimate}.
\end{proof}
We can now establish the main proposition of this part.  We recall that $\left(\hat{\rho}, \hat{u} \right)$ is a steady solution of \eqref{NS}-\eqref{BC}. We introduce the Evans function associated to $\left(\hat{\rho}, \hat{v} \right)$
\begin{equation*}
\mathcal{D} \left[\rho_{0},u_{0},u_{1} \right] (\lambda) = v(1),
\end{equation*}
where $\left(\rho,v \right)$ satisfies the ordinary differential equation 
\begin{equation}\label{eq_diff_evans_func}
\left\{
\begin{aligned}
&r_{x} = -\frac{\lambda r + \left( \hat{\rho} v \right)_{x} + \hat{u}_{x} r}{\hat{u}},\\
&\nu v_{xx} = \lambda \hat{\rho} v + \left(\hat{\rho} \hat{u} v \right)_{x} + P''(\hat{\rho}) \hat{\rho}_{x} r - P'(\hat{\rho}) \frac{\lambda r + \left( \hat{\rho} v \right)_{x} +\hat{u}_{x} r}{\hat{u}}  + \hat{u}_{x} \left(\hat{u} r + \hat{\rho} v \right)
\end{aligned}
\right.
\end{equation}
with
\begin{equation*}
r(0) = v(0) = 0 \text{ ,  }  v'(0)=1.
\end{equation*}
One can easily show that $\mathcal{D} \left[\rho_{0},u_{0},u_{1} \right](\lambda)=0$ if and only if $\lambda$ is an eigenvalue of \eqref{linear_steady}. We now establish the spectral stability of almost constant steady solutions of \eqref{NS}.

\begin{prop}
Assume that $P$ satisfies \eqref{pressure_cond}. Let $\left(\hat{\rho},\hat{u} \right)$ be the unique steady solution of \eqref{NS}-\eqref{BC}. Let $\epsilon>0$ be small enough. Then an eigenvalue $\lambda$ of \eqref{linear_steady}-\eqref{BC_linear_steady} has a negative real part.
\end{prop}

\begin{proof}
By Proposition \ref{high_freq_estimate}, problem \eqref{linear_steady}-\eqref{BC_linear_steady} does not have any eigenvalue of nonnegative real part outside a compact set $K$. Furthermore, from Proposition \ref{spec_stab_constant}, $\mathcal{D}\left[\rho_{0},u_{0},u_{0} \right]$ does not have any zero inside $K \cap \left\lbrace \Re > 0  \right\rbrace$. Since the Evans function $\mathcal{D}$ depends continuously on the boundary conditions, $\mathcal{D}\left[\rho_{0},u_{0},u_{1} \right]$ never vanishes inside $K \cap \left\lbrace \Re > 0  \right\rbrace$  for $\epsilon$ small enough.
\end{proof}

\subsection{About general steady states}

In the previous part, we only prove the spectral stability of almost constant states. In this part, we show some theoretical and numerical arguments that support the spectral stability of any steady states. 

\medskip

We know from previous works that the stability index criterion is a necessary condition for the spectral stability (see for instance \cite{gardner_jones_stab_ind,pego_weinstein_stab_ind,gardner_zumbrun_gap_lemma}). The stability index criterion states that
\begin{equation*}
\sgn \left(\mathcal{D} \left[\rho_{0},u_{0},u_{1} \right](0) \right) \sgn \left( \mathcal{D} \left[\rho_{0},u_{0},u_{1} \right](+ \infty) \right) = 1.
\end{equation*}
The following proposition shows that this criterion is satisfied.

\begin{prop}
For all steady states of problem \eqref{NS}-\eqref{BC}, the stability index criterion is satisfied.
\end{prop}

\begin{proof}
First we compute $\sgn \left(\mathcal{D} \left[\rho_{0},u_{0},u_{1} \right](0) \right)$. Proceeding as in Proposition \ref{compact_inverse_L}, we get the following system
\begin{equation*}
\left\{
\begin{aligned}
&\hat{\rho} v +\hat{u} r  = 0,\\
&\hat{\rho} \hat{u} v + \frac{\hat{\rho}}{\hat{u}} P'(\hat{\rho}) v = \nu v_{x} - \nu  v_{x}(0),\\
&r(0)=v(0)=0 \text{  ,  } v_{x}(0)=1,
\end{aligned}
\right.
\end{equation*}

\noindent and we obtain
\begin{equation*}
v(x) = \int_{0}^{x} \exp \left( \frac{1}{\nu} \int_{y}^{x}  \hat{\rho} \hat{u} - \frac{P'(\hat{\rho})}{\hat{u}} \hat{\rho} dz \right) dy \; v_{x}(0).
\end{equation*}

\noindent Then $\sgn v(1) = \sgn v_{x}(0) = 1$. Secondly, we compute $\mathcal{D} \left[\rho_{0},u_{0},u_{1} \right](+ \infty)$. We have
\begin{equation}\label{high_freq_system}
\left\{
\begin{aligned}
&\lambda r + \hat{u} r_{x} = f,\\
& \nu v_{xx} = \lambda \hat{\rho} v + P'(\hat{\rho}) r_{x} + g,\\
&r(0)=v(0)=0 \text{  ,  } v_{x}(0)=1,
\end{aligned}
\right.
\end{equation}
where $\lver f \rver_{2} + \lver g \rver_{2} \leq C \left(\lver r \rver_{2} + \lver v \rver_{2} + \lver v_{x} \rver_{2} \right)$. By solving the first equation of system \eqref{high_freq_system} we get for $\lambda$ large enough
\begin{equation*}
\lver r \rver_{2} \leq \frac{C}{\sqrt{\lambda}} \left(\lver v \rver_{2} + \rver v_{x} \lver_{2} \right).
\end{equation*}

\noindent We can rewrite the second equation of system \eqref{high_freq_system} as
\begin{equation*}
\nu v_{xx} = \lambda \hat{\rho} v + P'(\hat{\rho}) r_{x} +  \widetilde{g} \text{   with   } v(0)=0 \text{  ,  } v_{x}(0)=1,
\end{equation*}
where $\lver \widetilde{g} \rver_{2} \leq C \left(\lver v \rver_{2} + \lver v_{x} \rver_{2} \right)$. Then, we consider for $s \in [0,1]$ the equation
\begin{equation*}
\nu w_{xx} = \lambda \left((1-s) \hat{\rho} + s \right) w + (1-s) P'(\hat{\rho}) r_{x} + (1-s) \widetilde{g} \text{   with   } w(0)=w(1)=0.
\end{equation*}

\noindent Multiplying by $w$ and integrating, we notice that when $\lambda$ is large enough the only solution of this equation  is $w = 0$. Therefore, for $\lambda$ large enough, we define $z$ solution of
\begin{equation*}
\nu z_{xx} = \lambda z \text{   with   } z(0)=0 \text{  ,  } z_{x}(0)=1.
\end{equation*}  

\noindent and  $v(1)$ and $z(1)$ agree in sign.  It follows that $\sgn v(1) = \sgn z(1) = \sgn z_{x}(0) = 1$.
\end{proof}

This proposition also shows that problem \eqref{linear_steady}-\eqref{BC_linear_steady} has an even number of nonstable eigenvalues, i.e. eigenvalues with a nonnegative real part (see  \cite{gardner_jones_stab_ind,pego_weinstein_stab_ind,gardner_zumbrun_gap_lemma}). 

\medskip

\noindent Thanks to Lemma \ref{high_freq_estimate}, we can numerically check that $\sigma(\mathcal{S}^{-1} \mathcal{L})$ does not contain nonstable eigenvalues. Such verifications have for instance been done on the whole line (see \cite{num_stab_zum}). 

\begin{figure}[!t]
\centering
\includegraphics[scale=0.3]{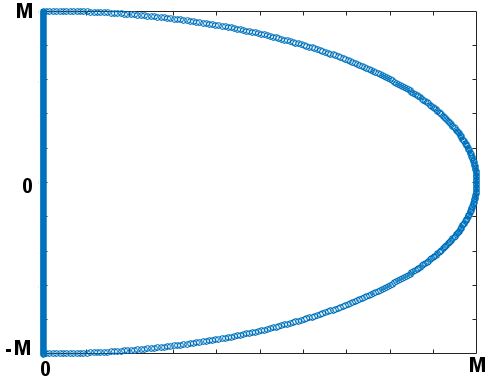}
\caption{Contour in the complex plane.}
\label{lambdagraph}
\end{figure}

\begin{figure}[!t]
\centering
\includegraphics[scale=0.35]{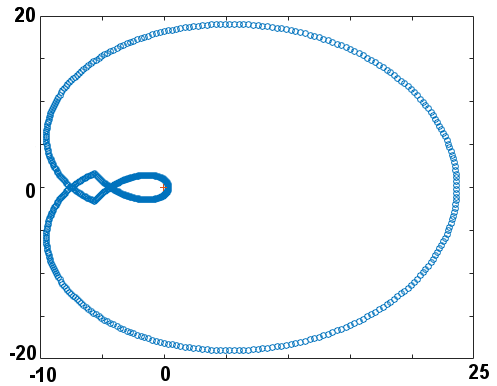}
\caption{Image of a contour mapped by the Evans function. $\nu=1$, $u_{0} = \frac{3}{2}$, $u_{1} = 1$, $\rho_{0} = 2$.}
\label{evanscontour}
\end{figure}

\medskip

In the following, we display some numerical simulations for a monatomic pressure law $P(\rho) = \rho^{1.4}$. For any $\lambda$, we can compute the associated Evans function thanks to system \eqref{eq_diff_evans_func}. We use a Runge Kutta 4 scheme. For each value of $u_{0}$, $u_{1}$, $\rho_{0}$ and $\nu$, we compute the Evans function along semi-circular contours of radius $M$ (see Figure \ref{lambdagraph}). We choose $M$ large enough such that our domain contains the half ball of Lemma \ref{high_frequency_control}. Figure \ref{evanscontour} represents the image of the contour with $M=10$, $\nu=1$, $u_{0} = \frac{3}{2}$, $u_{1} = 1$ and $\rho_{0} = 2$. Figure \ref{evanscontour2} represents the image of the contour with $M=10$, $\nu=0.1$, $u_{0} = \frac{3}{2}$, $u_{1} = 1$ and $\rho_{0} = 2$. We can see on these examples that the winding number of these graphs are both zero. Several computations have been performed for other values of the parameters $\nu \in [0.1,10]$, $u_{0} \in [1,10]$, $u_{1} \in [1,10]$ and $\rho_{0} \in [1,10]$. We could not find any nonstable eigenvalues.

\begin{figure}[!t]
\centering
\includegraphics[scale=0.35]{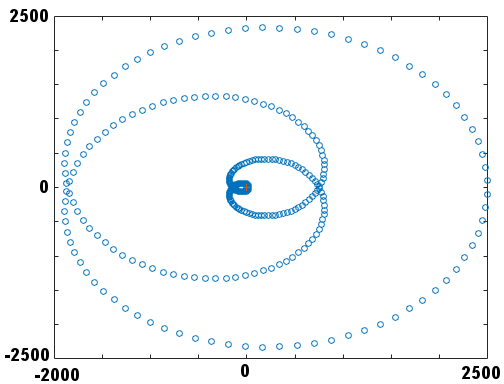}
\caption{Image of a contour mapped by the Evans function. $\nu=0.1$, $u_{0} = \frac{3}{2}$, $u_{1} = 1$, $\rho_{0} = 2$.}
\label{evanscontour2}
\end{figure}

\section{Local existence}\label{section_local_existence}

In this section, we state a local wellposedness result for problem \eqref{NS}-\eqref{BC} (see e.g. \cite{Matsumura_Nishida,Matsumura_Nishihara}).

\begin{prop}\label{local_existence}
Let $\rho_0>0,$ $u_0>0$ and $u_1>0$.  Assume that $P$ satisfies \eqref{pressure_cond}.  Let $\left(\rho_{ini},u_{ini} \right) \in {H^{1}}$ satisfying the boundary conditions \eqref{BC} and $\rho_{ini} > 0$. Then, there exists a time $T>0$ such that problem \eqref{NS}-\eqref{BC} has a unique solution $(\rho,u)$ in $\mathcal{C}\left([0,T];H^{1}(0,1) \right)$ with
\begin{equation*}
\underset{[0,T]}{\sup} \lver \left( \rho,u \right)(t) \rver_{H^{1}} \leq 2 \lver \left( \rho_{ini},u_{ini} \right) \rver_{H^{1}} \text{  and  } \rho(t,x) \geq \frac{\rho_{ini}(x)}{2} \text{ , } 0 \leq t \leq T \text{ , } 0 \leq x \leq 1.
\end{equation*}
\end{prop}

\section{Nonlinear stability}\label{section_nonlinear_stab}
For a solution $(\rho,u)$ of problem \eqref{NS}-\eqref{BC}, we define $(r,v) = (\rho - \hat{\rho}, u - \hat{u})$. We notice that $(r,v)$ satisfies the boundary conditions \eqref{BC_linear_steady} and ($\mathcal{L}$ and $\mathcal{S}$ are defined in Section \ref{eigenvalue_problem_part})

\begin{equation}\label{eq_nonlinear_int}
\begin{pmatrix} 1 & 0 \\ 0 &  \rho \end{pmatrix} \begin{pmatrix} r \\ v \end{pmatrix}_{t} - \mathcal{L} \begin{pmatrix} r \\ v \end{pmatrix} =  \begin{pmatrix} - \left( rv \right)_{x}  \\  -(\hat{u}v)_{x} r - \hat{\rho} v v_{x} - vrv_{x} - \left(P(\rho) - P(\hat{\rho}) - P'(\hat{\rho})r \right)_{x} \end{pmatrix}.
\end{equation}
Then, we get
\begin{equation}\label{eq_nonlinear}
\mathcal{S} \begin{pmatrix} r \\ v \end{pmatrix}_{t} - \mathcal{L} \begin{pmatrix} r \\ v \end{pmatrix} =  \mathcal{N} \text{  with  } r(t,0) = u(t,0) = u(t,1)=0 
\end{equation}
where $\mathcal{N}_{1} = -(rv)_{x}$ and
\begin{align*}
\mathcal{N}_{2} = &-\frac{\hat{\rho}}{\hat{\rho} + r} \left[ (\hat{u}v)_{x} r + \hat{\rho} v v_{x} + vrv_{x} + \left(P(r + \hat{\rho}) - P(\hat{\rho}) - P'(\hat{\rho})r \right)_{x} \right]\\
&+ \frac{r}{\hat{\rho} + r} \left[\left(\hat{\rho} \hat{u} v + P'(\hat{\rho}) r \right)_{x} + \hat{u}_{x} \left(\hat{u} r + \hat{\rho} v \right) - \nu v_{xx} \right].
\end{align*}
Notice that $\mathcal{N}_{1}(t,0)  = \mathcal{N}_{2}(t,0) = 0$ and that
\small
\begin{equation*}
\mathcal{N}_{2}(t,1) = - \left[ \hat{u} v_{x} r + \left(P'(\hat{\rho}+r) - P'(\hat{\rho}) - P''(\hat{\rho})r \right) \left(\hat{\rho} + r \right)_{x} + P''(\hat{\rho}) r r_{x} \right](t,1).
\end{equation*}
\normalsize
The following proposition is a nonlinear damping estimate.
\begin{prop}\label{nonlinear_damping}
Let $T>0$ and consider a solution $(r,v) \in \mathcal{C}\left([0,T];H^{1} \right)$ of \eqref{eq_nonlinear} on $[0,T]$. Assume that $P$ satisfies \eqref{pressure_cond} and that there exists $\epsilon >0$ small enough such that
\begin{equation*}
\underset{[0,T]}{\sup} \lver (r,v)(t) \rver_{H^{1}} \leq \epsilon.
\end{equation*}
Then, there exists some constants $C>0$ and $\theta_{0} > 0$ such that for all $0 \leq t \leq T$ and any $\theta \leq \theta_{0}$,
\begin{equation*}
\lver (r,v)(t) \rver_{H^{1}}^{2} \leq C e^{-\theta t} \lver (r,v)(0) \rver_{H^{1}}^{2} + C \int_{0}^{t} e^{-\theta (t-s)} \lver (r,v)(s) \rver_{2}^{2} ds.
\end{equation*}
Furthermore, if $(r,v) \in \mathcal{C}\left([0,T];H^{2} \times H^{3} \right) \cap \mathcal{C}^{1}\left([0,T];H^{1} \right)$ and for $\epsilon$ small enough
\begin{equation*}
\underset{[0,T]}{\sup} \lver (r,v)(t) \rver_{H^{2} \times H^{3}} \leq \epsilon,
\end{equation*}
there exists some constants $C>0$ and $\theta_{1} > 0$, for any $0 \leq t \leq T$ and $\theta < \theta_{1}$
\begin{equation*}
\lver (r,v)(t) \rver_{H^{2} \times H^{3}}^{2} \leq C e^{-\theta t} \lver (r,v)(0) \rver_{H^{2} \times H^{3}}^{2} + C \int_{0}^{t} e^{-\theta (t-s)} \lver (r,v)(s) \rver_{2}^{2} ds.
\end{equation*}
\end{prop}

\begin{proof}
This proof is based on an appropriate Goodman-type energy estimate and is similar to the proof of Proposition \ref{high_freq_estimate}. We define the energy equivalent to the $H^{1}$-norm (by the Poincar\'e inequality \ref{poincare})
\begin{equation*}
\mathcal{E} \left(r,v \right) = \frac{1}{2} \int_{0}^{1} \phi_{1} \lver r_{x} \rver^{2} + \phi_{2} \lver \left( \hat{\rho} v \right)_{x} \rver^{2}
\end{equation*}
where $\phi_{1}$ and $\phi_{2}$ satisfy
\begin{equation*}
\phi_{1} > 0 \text{ ,  } \phi_{2} > 0 \text{ ,  }   \phi_{1} = P'(\hat{\rho}) \phi_{2} \text{ ,  } \frac{1}{2} (\hat{u} \phi_{1})_{x} - 2 \hat{u}_{x} \phi_{1} < 0.
\end{equation*}

Then, after some computations, we obtain
\small
\begin{equation*}
\begin{aligned}
\frac{d}{dt} \mathcal{E} \left(r,v \right) \leq &-\nu \int_{0}^{1} \hat{\rho} \phi_{2} \lver v_{xx} \rver^{2} + \left(\frac{1}{2} (\hat{u} \phi_{1})_{x} - 2 \hat{u}_{x} \phi_{1} \right) \lver r_{x} \rver^{2} + r_{x} v_{xx} \left(- \hat{\rho} \phi_{1} + \hat{\rho} P'(\hat{\rho}) \phi_{2} \right)\\
&+ \left[-\frac{1}{2} \hat{u} \phi_{1} \lver r_{x} \rver^{2} + \phi_{2} \hat{\rho} v_{x} \left(\nu v_{xx} - P'(\hat{\rho}) r_{x} \right) \right]_{0}^{1} + \int_{0}^{1} \phi_{1} r_{x} \left( \mathcal{N}_{1} \right)_{x} + \phi_{2} \left( \hat{\rho} v \right)_{x} \left(\mathcal{N}_{2} \right)_{x}\\
&+C \lver r_{x} \rver_{2} \left( \lver r \rver_{2} + \lver v \rver_{H^{1}} \right) + C \lver v \rver_{H^{1}} \left( \lver r \rver_{H^{1}} + \lver v \rver_{H^{1}} + \lver v_{xx} \rver_{2}\right).
\end{aligned}
\end{equation*}
\normalsize
Integrating by parts and using Lemma \ref{Linfty_controls} we get
\begin{equation*}
\int_{0}^{1} \phi_{1} r_{x} \left( \mathcal{N}_{1} \right)_{x} + \phi_{2} \left( \hat{\rho} v \right)_{x} \left(\mathcal{N}_{2} \right)_{x} \leq \left[\phi_{2} \hat{\rho} v_{x} \mathcal{N}_{2} \right]_{0}^{1} + C \lver (r,v) \rver_{H^{1}}^{2} (\lver (r,v) \rver_{H^{1}} + \lver v_{xx} \rver_{2}).
\end{equation*} 
Then, since $\mathcal{N}_{1}(t,0) = \mathcal{N}_{2}(t,0) = 0$, we have
\small
\begin{align*}
& u_{0} r_{x}(t,0) = - \rho_{0} v_{x}(t,0),\\
&\nu v_{xx}(t,0) - P'(\rho_{0}) r_{x}(t,0) = \rho_{0} u_{0} v_{x}(t,0),\\
&\nu v_{xx}(t,1) - P'(\hat{\rho}(1)) r_{x}(t,1) + \mathcal{N}_{2}(t,1) = \hat{\rho}(1) u_{1} v_{x}(t,1) + P''(\hat{\rho}(1)) \hat{\rho}_{x}(1) r(t,1) + \hat{u}_{x}(1) u_{1} r(t,1).
\end{align*}
\normalsize
Finally, thanks to the previous boundary equalities, Lemma \ref{Linfty_controls}, Lemma \ref{interpolation_thm}, Young's inequality and the fact that $\lver (r,v) \rver_{H^{1}}$ is small enough, we obtain
\begin{equation*}
\frac{d}{dt} \mathcal{E} \left(r,v \right) \leq - \theta_{0} \lver (r_{x},v_{xx}) \rver_{2}^{2} + C \lver (r,v) \rver_{2}^{2}.
\end{equation*}
The first inequality easily follows from the Poincar\'e-Wirtinger inequality. Similarly, since $(r_{t},v_{t})$ satisfies the boundary conditions \eqref{BC_linear_steady}, we get for $\epsilon$ and $\theta_{0}$ small enough
\begin{equation*}
\frac{d}{dt} \mathcal{E} \left(r_{t},v_{t} \right) \leq - \theta_{0} \lver (r_{tx},v_{txx}) \rver_{2}^{2} + C\lver (r_{t},v_{t}) \rver_{2}^{2}.
\end{equation*}
Then, using \eqref{eq_nonlinear_int} we notice that
\begin{equation*}
\lver (r_{t},v_{t}) \rver_{2} \leq  C \lver (r,v) \rver_{H^{1}} + C \lver v_{xx} \rver_{2}.
\end{equation*}
Therefore using the Poincar\'e inequalities, we get for $\delta$ and $\theta_{1}$ small enough
\begin{equation}\label{Et}
\frac{d}{dt} \left( \mathcal{E} \left( r,v \right) + \delta \mathcal{E} \left( r_{t},v_{t} \right) \right)  \leq - \theta_{1} \left( \lver (r_{x},v_{x}) \rver_{2}^{2} + \lver (r_{tx},v_{tx}) \rver_{2}^{2} \right) + C \lver (r,v) \rver_{2}^{2}
\end{equation}
and the result easily follows from the fact that
\begin{equation*}
\lver (r,v) \rver_{H^{2} \times H^{3}} \leq  C \lver (r,v) \rver_{H^{1}} + C \lver (r_{t},v_{t}) \rver_{2}.
\end{equation*}
\end{proof}

\begin{remark}\label{NZrmk}
By taking further time-derivatives, we could obtain an estimate similar to \eqref{Et} in an arbitrarily high-regularity Sobolev
space of mixed type $H^r\times H^s$, with $s\sim 2r$ as $r\to \infty$.
This observation repairs a minor error in \cite{toan_zumbrun_ns}, citing an estimate with $r=s$.
\end{remark}

We can now state the main result of this paper. 

\begin{thm}\label{stab_result}
Let $\rho_0>0,$ $u_0>0$ and $u_1>0$. Let $(\hat{\rho},\hat{u})$ be the unique steady solution of problem \eqref{NS}-\eqref{BC}. Assume that $P$ satisfies \eqref{pressure_cond}. Assume that there exists $\alpha>0$ such that $\Re (\sigma(\mathcal{S}^{-1} \mathcal{L})) < - \alpha$. Then, there exists  $\epsilon>0$ and $\theta > 0$, for any $\left(\rho_{ini},u_{ini} \right) \in H^{2} \times H^{3}$ satisfying the boundary conditions \eqref{BC}, the compatibility conditions
\small
\begin{equation}\label{compatibility_bound_cond}
\left(\rho_{ini} u_{ini} \right)_{x}\!(0)=0, \left(\rho_{ini} u_{ini}^{2} + P(\rho_{ini}) -  \nu u_{ini \; x}  \right)_{x}\!(0)=0, \left(\rho_{ini} u_{ini}^{2} + P(\rho_{ini}) -  \nu u_{ini \; x}  \right)_{x}\!(1)=0
\end{equation}
\normalsize
and 
\begin{equation*}
\lver \left(\rho_{ini},u_{ini} \right) - \left(\hat{\rho},\hat{u} \right) \rver_{H^{2} \times H^{3}}  \leq \epsilon,
\end{equation*}
the unique solution $(\rho,u)$ of problem \eqref{NS}-\eqref{BC} with the initial condition $\left(\rho_{ini},u_{ini} \right)$ satisfies
\begin{equation*}
\lver \left(\rho,u \right)(t) - \left(\hat{\rho},\hat{u} \right) \rver_{H^{2} \times H^{3}} \leq C \lver \left(\rho_{ini},u_{ini} \right) - \left(\hat{\rho},\hat{u} \right) \rver_{H^{2} \times H^{3}} e^{- \theta t}.
\end{equation*}
\end{thm}

\begin{remark}\label{H1XL2_explanation}
As we will see in the proof, since we do not know if $\mathcal{N}_{2}(t,1) = 0$, the only way to use a linear damping estimate is to work in $L^{2}$ for the $v$ component. That is why in Proposition \ref{pruss_thm} we used $H^{1} \times L^{2}$ and not $H^{1}$. Notice also that we impose the compatibility conditions \eqref{compatibility_bound_cond} in order to get enough regularity.
\end{remark}

\begin{proof}
We denote by $U(t,x) = \left(\rho,u \right)(t) - \left(\hat{\rho},\hat{u} \right)(t,x)$. Let $T$ be the existence time of Proposition \ref{local_existence}. The Duhamel formulation of Equation \eqref{eq_nonlinear} is, for $0\leq t \leq T$,
\begin{equation*}
U(t) = e^{t \mathcal{S}^{-1} \mathcal{L}} U(0) + \int_{0}^{t} e^{(t-s) \mathcal{S}^{-1} \mathcal{L}} \mathcal{S}^{-1} \mathcal{N}(s) ds.
\end{equation*}
Noticing that $\mathcal{N} \in \left\{(r,v) \in H^{1} \times L^{2} \text{, } r(0)=0 \right\}$ and that $\mathcal{N}$ contains at least quadratic terms, Proposition \ref{pruss_thm} gives the existence of $\theta>0$
\begin{equation*}
\lver U(t) \rver_{2} \leq \lver U(t) \rver_{H^{1} \times L^{2}} \leq C e^{- \theta t} \lver U(0) \rver_{H^{1} \times L^{2}} + \int_{0}^{t} e^{- \theta (t-s)} \lver U(s) \rver_{H^{2}}^{2} C \left(\lver U(s) \rver_{H^{2}} \right) ds.
\end{equation*}
Then, the equality \eqref{eq_nonlinear_int} gives 
\begin{equation*}
\lver U(t) \rver_{2} \leq C e^{- \theta t} \lver U(0) \rver_{H^{1} \times L^{2}} + \int_{0}^{t} C \left(\lver U(s) \rver_{H^{2}} \right) e^{- \theta (t-s)} \left( \lver U(s) \rver_{H^{1}}^{2} + \lver U_{t}(s) \rver_{H^{1}}^{2} \right) ds.
\end{equation*}
Furthermore, the compatibility conditions \eqref{compatibility_bound_cond} imply that $U \in \mathcal{C}\left([0,T];H^{2} \times H^{3} \right)$. Therefore, we can use the nonlinear damping estimate of Proposition \ref{nonlinear_damping} and by Proposition \ref{local_existence} the $H^{2}$-norm of $U$ is controlled by the initial condition. We get for $\epsilon$ and $\theta>0$ small enough
\begin{equation*}
\lver U(t) \rver_{2} \leq C \left(\lver U(0) \rver_{H^{2} \times H^{3}} \right) \left((1+t) e^{- \theta t} \lver U(0) \rver_{H^{2} \times H^{3}} + \int_{0}^{t} (t-s) e^{- \theta (t-s)} \lver U(s) \rver_{2}^{2} ds \right).
\end{equation*}
Denoting $\zeta_{0}(t) = \underset{[0,t]}{\sup} \;\; e^{\frac{\theta}{2} s} \lver U(s) \rver_{2}$, we obtain that for $0 \leq t \leq T$
\begin{equation*}
\zeta_{0}(t) \leq C \left( \lver U(0) \rver_{H^{2} \times H^{3}} \right) \left( \epsilon + \zeta_{0}(t)^{2} \right). 
\end{equation*}
Furthermore, denoting $\zeta_{1}(t) = \underset{[0,t]}{\sup} \;\; \left( e^{\frac{\theta}{2} s} \lver U(s) \rver_{H^{1}} + e^{\frac{\theta}{2} s} \lver U_{t}(s) \rver_{H^{1}} \right)$ and using Proposition \ref{nonlinear_damping}, $\zeta_{1}$ is also controlled on $[0,T]$. Finally, if $\epsilon$ is small enough, we can take $T = +\infty$ and $\zeta_{1}$ is bounded on $\mathbb{R}^{+}$.
\end{proof}

\section{An improvement in some situations}\label{section_improvement}

The main result of this paper, Theorem \ref{stab_result}, states that spectrally stable steady states are stable in $H^{2} \times H^{3}$. In this part, we prove that under more restrictive conditions, we can state a stability result in $H^{1} \times H^{2}$. To achieve that, we add another assumption
\begin{equation}\label{cond2}
\begin{aligned}
&P'' > 0 \text{ if } \hat{u}_{x} > 0 \text{   (compressive solutions)},\\ 
&\frac{P''(y)}{P'(y)} < \frac{2}{y} \text{ and } \hat{\rho}_{x} < \frac{1}{4} \hat{\rho} \text{ if } \hat{u}_{x} < 0 \text{  (small expansive solutions)}.\\
\end{aligned}
\end{equation}
With this additional assumption, we can establish a high frequency estimate in $L^{2}$.
\begin{prop}\label{high_freq_estimate_L2}
Assume that $P$ satisfies \eqref{pressure_cond} and that Condition \eqref{cond2} is satisfied. There exists a constant $\alpha > 0$ such that if $\Re(\lambda) > - \alpha$ and $\lver \lambda \rver$ is large enough,
\begin{equation*}
\lver (\rho,v) \rver^{2}_{2} \leq C \lver \left(\lambda  \mathcal{S} - \mathcal{L} \right) (\rho,v) \rver^{2}_{2}, 
\end{equation*}
for any $(\rho,v)$ satisfying the boundary conditions \eqref{BC_linear_steady}. 
\end{prop}

\begin{proof}
This proof is based on an appropriate Goodman-type energy estimate. In the following we denote $ \left(\lambda \mathcal{S} - \mathcal{L}\right) (\rho,v) =(f,g)$. We define the following energy 
\begin{equation*}
\mathcal{E} \left(r,v \right) = \frac{1}{2} \int_{0}^{1} \phi_{1} \lver r \rver^{2} + \phi_{2} \hat{\rho} \lver  v \rver^{2}
\end{equation*}
where $\phi_{1}$ and $\phi_{2}$ satisfy
\begin{align*}
\phi_{1} > 0 \text{ ,  } \phi_{2} > 0 \text{ ,  }  \hat{\rho} \phi_{1} = P'(\hat{\rho}) \phi_{2}.
\end{align*}
Then, we compute 
\begin{equation*}
2 \Re(\lambda) \mathcal{E} \left(r,v \right) = \Re \left( \int_{0}^{1} \phi_{1} \overline{r} \lambda r \right) + \Re \left( \int_{0}^{1} \phi_{2} \hat{\rho} \overline{v} \lambda v \right).
\end{equation*}
After some computations we get
\small
\begin{equation*}
\begin{aligned}
2 \Re(\lambda) \mathcal{E} \left(r,v \right) \leq &\int_{0}^{1} - \nu \phi_{2} \lver v_{x} \rver^{2} + \frac{1}{2} \hat{u}^{2} \left( \frac{\phi_{1}}{\hat{u}} \right)_{x} \lver r \rver^{2} + \Re(\overline{r} v_{x}) \left[P'(\hat{\rho}) \phi_{2} - \hat{\rho} \phi_{1} \right] + C \lver (f,g) \rver_{2} \lver (r,v) \rver_{2}\\
&+ \int_{0}^{1} \left[\nu (\phi_{2})_{xx} - 2 \phi_{2} \hat{u}_{x} \hat{\rho} + (\phi_{2})_{x} \hat{\rho} \hat{u} \right]  \frac{\lver v \rver^{2}}{2} + \Re \left(v \overline{r} \right) \left[(\phi_{2})_{x} P'(\hat{\rho}) - \phi_{1} \hat{\rho}_{x} - \phi_{2} \hat{u}_{x} \hat{u} \right]\!.
\end{aligned}
\end{equation*}
\normalsize
Then, we separately consider the three situations $\hat{u}_{x} > 0$ (compressive solution), $\hat{u}_{x} = 0$ (constant solution) and $\hat{u}_{x} < 0$ (expansive solution).
\medskip

\noindent - If $\hat{u}_{x} > 0$, we take $\phi_{2} = 1$, $\phi_{1} = \frac{P'(\hat{\rho})}{\hat{\rho}}$ and we get

\begin{equation*}
\hat{u}^{2} \left( \frac{\phi_{1}}{\hat{u}} \right)_{x} = \hat{u} (\phi_{1})_{x} - \hat{u}_{x} \phi_{1} = \frac{P''(\hat{\rho}) \hat{\rho}_{x} \hat{u}}{\hat{\rho}} < 0 \text{  ,  } \nu (\phi_{2})_{xx} - 2 \phi_{2} \hat{u}_{x} \hat{\rho} + (\phi_{2})_{x} \hat{\rho} \hat{u} = -2 \hat{u}_{x} \hat{\rho} < 0.
\end{equation*}
\medskip

\noindent - If $\hat{u}_{x} = 0$, we take $\phi_{1} = P'(\hat{\rho}) - \beta x$, $\phi_{2} = \hat{\rho} - \beta \frac{\hat{\rho}}{P'(\hat{\rho})} x$ with $\beta > 0$ small enough, and we get

\begin{equation*}
\hat{u} (\phi_{1})_{x} - \hat{u}_{x} \phi_{1} = - \beta \hat{u} < 0 \text{  ,  } \nu (\phi_{2})_{xx} - 2 \phi_{2} \hat{u}_{x} \hat{\rho} + (\phi_{2})_{x} \hat{\rho} \hat{u} =  - \beta \frac{\hat{\rho}^{2} \hat{u}}{P'(\hat{\rho})}  < 0.
\end{equation*}
\medskip

\noindent - If $\hat{u}_{x} < 0$, we take $\phi_{2}(x) = \sqrt{M-2x}$, $M>2$ and $\phi_{1}(x) = \frac{P'(\hat{\rho})}{\hat{\rho}} \phi_{2}(x)$ and thanks to Condition \eqref{cond2} we get

\begin{align*}
&\hat{u} (\phi_{1})_{x} - \hat{u}_{x} \phi_{1} = \frac{\phi_{2}}{\hat{\rho}} P'(\hat{\rho}) \hat{u} \left( \frac{P''(\hat{\rho})}{P'(\hat{\rho})} \hat{\rho}_{x} - \frac{1}{M-2x} \right) < 0,\\
&\nu (\phi_{2})_{xx} - 2 \phi_{2} \hat{u}_{x} \hat{\rho} + (\phi_{2})_{x} \hat{\rho} \hat{u} \leq \phi_{2} \hat{u} \left(2 \hat{\rho}_{x} - \frac{\hat{\rho}}{M-2x} \right)  < 0.
\end{align*}

\medskip

\noindent Moreover, in any case, we have (denoting $\tilde{r}(x) = \int_{0}^{x} r(y) dy$)

\begin{equation*}
\int_{0}^{1} \Re \left(v \overline{r} \right) \left((\phi_{2})_{x} P'(\hat{\rho}) - \phi_{1} \hat{\rho}_{x} - \phi_{2} \hat{u}_{x} \hat{u} \right) \leq C \lver \tilde{r} \rver_{2} \lver v \rver_{H^{1}}.
\end{equation*}
\noindent Thus, using the first inequality of Lemma \ref{high_frequency_control}, we can find a constant $\alpha > 0$, for $\lver \lambda \rver$ large enough,
\begin{equation*}
2 \Re(\lambda) \mathcal{E} \left(r,v \right) \leq - \alpha \lver (r,v) \rver_{2}^{2} + C \lver \left(\lambda - \mathcal{L}\right) (\rho,v) \rver_{2}^{2},
\end{equation*}
and the inequality follows. 
\end{proof}

Thanks to this $L^2$ high frequency estimate, we can improve Proposition \ref{pruss_thm}. Under the assumption that $Re \left( \sigma(\mathcal{S}^{-1} \mathcal{L} \right)) \leq - \alpha < 0$, we get
\begin{equation*}
\lver e^{t \mathcal{S}^{-1} \mathcal{L}} (r,v) \rver_{2} \leq C e^{-\alpha  t} \lver (r,v) \rver_{2}.
\end{equation*}
Furthermore, thanks to the previous appropriate Goodman-type estimate, we can improve the nonlinear damping estimate in Proposition \ref{nonlinear_damping}. If $(r,v)$ in $\mathcal{C}\left([0,T];H^{2} \times H^{1} \right)$ is a solution of \eqref{eq_nonlinear} on $[0,T]$ and 
\begin{equation*}
\underset{[0,T]}{\sup} \lver (r,v)(t) \rver_{H^{1} \times H^{2}} \leq \epsilon
\end{equation*} 
for $\epsilon$ small enough, we have
\begin{equation*}
\lver (r_{t},v_{t})(t) \rver_{2}^{2} \leq C e^{-\theta t} \lver (r,v)(0) \rver_{H^{1} \times H^{2}}^{2} + C \int_{0}^{t} e^{-\theta (t-s)} \lver (r,v)(s) \rver_{2}^{2} ds.
\end{equation*} 
Finally, applying the Duhamel formulation in $L^{2}$, we obtain the following theorem.
\begin{thm}\label{stab_result2}
Let $\rho_0>0,$ $u_0>0$ and $u_1>0$. Let $(\hat{\rho},\hat{u})$ be the unique steady solution of problem \eqref{NS}-\eqref{BC}. Assume that $P$ satisfies \eqref{pressure_cond} and that Condition \eqref{cond2} is satisfied. Assume that there exists $\alpha>0$ such that $\Re (\sigma(\mathcal{S}^{-1} \mathcal{L})) < - \alpha$. Then, there exists  $\epsilon>0$ and $\theta > 0$, for any $\left(\rho_{ini},u_{ini} \right) \in H^{1} \times H^{2}$ satisfying the boundary conditions \eqref{BC} and
\begin{equation*}
\lver \left(\rho_{ini},u_{ini} \right) - \left(\hat{\rho},\hat{u} \right) \rver_{H^{1} \times H^{2}}  \leq \epsilon,
\end{equation*}
the unique solution $(\rho,u)$ of problem \eqref{NS}-\eqref{BC} with the initial condition $\left(\rho_{ini},u_{ini} \right)$ satisfies
\begin{equation*}
\lver \left(\rho,u \right)(t) - \left(\hat{\rho},\hat{u} \right) \rver_{H^{1} \times H^{2}} \leq C \lver \left(\rho_{ini},u_{ini} \right) - \left(\hat{\rho},\hat{u} \right) \rver_{H^{1} \times H^{2}} e^{- \theta t}.
\end{equation*}
\end{thm}

\appendix 
\section{$L^{\infty}$ estimates and interpolation}\label{section_appendix}

In this appendix, we recall some basic results about Sobolev spaces in a bounded domain. The first lemma is a Poincar\'e inequality.
\begin{lemma}\label{poincare}
For any $f \in H^{1}(0,1)$ with $f(0)=0$, we have 
\begin{equation*}
\lver f \rver_{2} \leq  2 \lver f_{x} \rver_{2}.
\end{equation*}
\end{lemma}
\begin{proof}
\noindent For any $x,y \in [0,1]$, we have
\begin{equation}\label{int_sobolev}
f(x)^{2} = f(y)^{2} + 2 \int_{y}^{x} f(z)f'(z)dz. 
\end{equation}
Since $f(0)=0$, we choose $y=0$ and we obtain the result by integrating over $x$.
\end{proof}

The following lemma allows us to control boundary terms and $L^{\infty}$-norms by appropriate Sobolev norms.
\begin{lemma}\label{Linfty_controls}
For any $f \in H^{1}(0,1)$, we have 
\begin{align*}
&\lver f \rver_{\infty} \leq \sqrt{2 \lver f \rver_{2} \lver f_{x} \rver_{2}} \text{  , if  } f(0)=0 \text{, }\\
&\lver f \rver_{\infty} \leq \lver f \rver_{2} + \sqrt{2 \lver f \rver_{2} \lver f_{x} \rver_{2}}.
\end{align*}
\end{lemma}
\begin{proof}
It is a direct consequence of the equality \eqref{int_sobolev}.
\end{proof}

We also have the following derivative-interpolation theorem.

\begin{lemma}\label{interpolation_thm}
For any $v \in H^{2}(0,1)$,
\begin{equation*}
\lver v_{x} \rver_{2}^{2} \leq C \lver v \rver_{2}^{2} +  C \lver v \rver_{2} \lver v_{xx} \rver_{2}.
\end{equation*}
\end{lemma}

\begin{proof}
Integrating by parts, we get
\begin{equation*}
\int_{0}^{1} v_{x}^{2} dx = -\int_{0}^{1} v v_{xx} dx + \left[ v v_{x} \right]_{0}^{1}.
\end{equation*}
The result follows from the previous lemma.
\end{proof}

\scriptsize
\bibliographystyle{alpha}
\bibliography{biblio}
\normalsize

\end{document}